\documentclass{commat}

\newcommand\C{\mathbb{C}}    
\newcommand\N{\mathbb{N}}    
\newcommand\Z{\mathbb{Z}}    
\newcommand\F{\mathbb{F}}    
\newcommand\R{\mathbb{R}}    
\newcommand\Q{\mathbb{Q}}    
\newcommand\algI{1}    
\newcommand\lket{\left|}    
\newcommand\rket{\right>}    
\newcommand\lbrak{\left[}    
\newcommand\rbrak{\right]}    
\newcommand\crea{{a^+}}    
\newcommand\pcrea{{(a^+)}}    
\newcommand\anni{a}    
\newcommand\Num{N}    
\newcommand\AW{AW(3)}    
\newcommand\AWq{AW_q(3)}    
\newcommand\HeisenAlg{\mathcal{H}}    
\newcommand\Heisen{\HeisenAlg(q)}    
\newcommand\gradsub{\mathfrak{H}}    
\newcommand\freq{\omega}    
\newcommand\algBound{\mathcal{B}(\ell_2)}    
\newcommand\genzero{K_0}    
\newcommand\genone{K_1}    
\newcommand\gentwo{K_2}    
\newcommand\oscLie{\lbrak\anni,\crea\rbrak}    
\newcommand\Proto{\mathcal{P}}    
\newcommand\Protoq{\mathcal{P}_q}    
\newcommand\dcmH{\eta}    
\newcommand\dcmTH{\theta}    
\newcommand\ALLdcm{\Phi}    
\newcommand\TheHom{\Psi}    
\newcommand\repHeisen{\HeisenAlg_0}    
\newcommand\repHeisenLie{\mathfrak{L}(\HeisenAlg_0)}    
\newcommand\repLieAB{{\lbrack\anni,\crea\rbrack}}    
\newcommand\nonLiesub{\Gamma}    
\newcommand\ad{\mbox{ad\ }}    

\newcommand\LieAB{\lbrak A,B\rbrak}    

\newcommand\creaIso{B}    
\newcommand\anniIso{A}    
\newcommand\LieABiso{\lbrak\anniIso,\creaIso\rbrak}    
\newcommand\pcreaIso{B}    
\newcommand\LieProto{\mathcal{L}(\Proto)}    
\newcommand\LieAWq{\mathcal{L}(\AWq)}    

\newcommand\Zplus{\N\backslash\{0\}}    

\newcommand\setdiff{\backslash}    
\newcommand\Zdcm{\ALLdcm_\Z}    
\newcommand\nonZdcm{\ALLdcm\setdiff\Zdcm}    

\newcommand\HeisenIs{\F\left<A,B\right>/(AB-qBA-1)}    

\newtheorem{assumption}[definition]{Assumption}

\title{%
    Extended commutator algebra for the $q$-oscillator\\ and a related Askey-Wilson algebra
    }

\author{%
    Rafael Reno S. Cantuba
    }

\authorinfo[%
    R. R. S. Cantuba]{
    De La Salle University, Philippines}{%
    rafael\_cantuba@dlsu.edu.ph
    }

\abstract{%
    Let $q$ be a nonzero complex number that is not a root of unity. In the $q$-oscillator with commutation relation $\anni\crea-q\crea\anni=1$, it is known that the smallest commutator algebra of operators containing the creation and annihilation operators $\crea$ and $\anni$ is the linear span of $\crea$ and $\anni$, together with all operators of the form $\crea^l\repLieAB^k$, and $\repLieAB^k\anni^l$, where $l$ is a nonnegative integer and $k$ is a positive integer. That is, linear combinations of operators of the form $\anni^h$ or $(\crea)^h$ with $h\geq 2$ or $h=0$ are outside the commutator algebra generated by $\anni$ and $\crea$. This is a solution to the Lie polynomial characterization problem for the associative algebra generated by $\crea$ and $\anni$. In this work, we extend the Lie polynomial characterization into the associative algebra $\Proto=\Proto(q)$ generated by $\anni$, $\crea$, and the operator $e^{\omega N}$ for some nonzero real parameter $\omega$, where $N$ is the number operator, and we relate this to a $q$-oscillator representation of the Askey-Wilson algebra $AW(3)$.
    }


\msc{%
    47L30, 17B60, 17B65, 16S15, 81R50
    (see \url{https://mathscinet.ams.org/mathscinet/msc/msc2020.html})
    }

\VOLUME{32}
\YEAR{2024}
\NUMBER{2}
\firstpage{1}
\DOI{https://doi.org/10.46298/cm.10820}

\begin{document}

\section{Introduction}

Throughout, we consider operators on the Hilbert space $\ell_2$ over the complex field $\C$, which consists of all sequences $\left(z_n\right)_{n=0}^\infty$ of elements of $\C$ with the property $ \sum_{n=0}^\infty \left|z_n\right|^2 <\infty $.\linebreak We use the complete orthonormal basis of $\ell_2$ consisting of the ket vectors $\lket n\rket $ for all\linebreak $n\in\N:=\{0,1,\ldots\}$. Denote by $\algBound$ the Banach algebra of all bounded operators on $\ell_2$. In the traditional harmonic oscillator, the creation operator $\crea$ and the annihilation operator $\anni$ can be represented as elements of $\algBound$ acting on the basis ket vectors according to the equations
\begin{eqnarray}
\anni \lket 0\rket &=&0,\nonumber\\
\anni \lket n\rket &=&\sqrt{n}\lket n-1\rket,\quad\quad\quad\ \ (n\in\N\backslash\{0\}),\nonumber\\
\crea \lket n\rket &=&\sqrt{n+1}\lket n+1\rket,\quad\quad (n\in\N).\nonumber
\end{eqnarray}
As a~consequence, $\crea$ and $\anni$ satisfy the canonical commutation relation
\begin{eqnarray}
\lbrak\anni,\crea\rbrak = 1,\label{classicComm}
\end{eqnarray}
where the operation $\lbrack -,- \rbrack : \algBound\times\algBound\rightarrow\algBound$ is the \emph{commutator} given by the rule $\lbrack X,Y\rbrack:=XY-YX$ for all $X,Y\in\algBound$. The algebra $\algBound$ is hence an abstract Lie algebra with respect to $\lbrack -,-\rbrack$ and the vector space operations.

\begin{problem}\label{classicProb} What is the smallest space of operators containing $\crea$ and $\anni$ that is closed under the commutator and vector space operations? In other words, viewing $\algBound$ as a~Lie algebra under said operations, what is the Lie subalgebra of $\algBound$ generated by $\crea$ and $\anni$? Equivalently, what is the \emph{commutator algebra} for the creation and annihilation operators of the classical harmonic oscillator?
\end{problem}

The solution to Problem~\ref{classicProb} is readily available from the theory of the classification of low-dimensional Lie algebras. As a~consequence of the relation~\eqref{classicComm}, the Lie subalgebra of $\algBound$ generated by $\crea$ and $\anni$ is isomorphic to the \emph{Heisenberg algebra}, which is the three-dimensional Lie algebra whose derived (Lie) algebra is contained in the center. The operators $\crea$, $\anni$ and $1\in\algBound$ form a~basis~\cite[Section 3.2.1]{Erd06}.

\subsection{The commutator algebra for \texorpdfstring{$\crea$, $\anni$}{a+, a} in the \texorpdfstring{$q$}{q}-oscillator}

Several modifications or generalizations to the relation~\eqref{classicComm} have been proposed. One is the replacement of the undeformed commutator by the $q$-deformed commutator\linebreak $\lbrack\anni,\crea\rbrack_q:=\anni\crea-q\crea\anni$ for some scalar parameter $q$. The result is the $q$-deformed harmonic oscillator or $q$-oscillator introduced in the classical works~\cite{Bie89,Mac89}. Another is the $\omega$-commutator $\lbrack\anni,\crea\rbrack_\omega:=e^\omega\anni\crea-e^{-\omega}\crea\anni$ for some nonzero real parameter $\omega$, which shall arise in later sections of this work. The latter deformed commutator also leads to a~$q$-oscillator representation~\cite{Zhe91}. See, for instance,~\cite{Bat15,Reg01} for physical interpretations or realizations of the $q$-oscillator. The significance of the $q$-deformed oscillator is that, together with the deformed quantum group $SU_q(2)$, these two mathematical constructions were once considered as possible candidates for the role of oscillator and angular momentum, respectively, at very small scales, such as the Planck scale~\cite[p. 1155]{Zhe91}. The natural continuation now is to extend Problem~\ref{classicProb} for the $q$-oscillator.

\begin{problem}\label{qProb} Given a~scalar $q\neq 1$, what is the commutator algebra for the creation and annihilation operators of the $q$-oscillator?
\end{problem}

A solution to Problem~\ref{qProb} has recently been obtained in~\cite{Can19a,Can19b}. Some generalizations and alternative approaches were studied in~\cite{Can20b,Can22,Can21,Can20a}. Before we describe the solution, we first discuss some consequences of, and some algebraic structures related to, the $q$-deformed commutation relation
\begin{eqnarray}\label{qComm}
\anni\crea-q\crea\anni = 1.
\end{eqnarray}
The \emph{$q$-deformed Heisenberg algebra} is the unital associative algebra $\Heisen$ (with unit $1$) abstractly defined by a~presentation with generators $A,B$ and relation $AB-qBA=1$. More succinctly, if the underlying field is $\F$, then $\Heisen:=\HeisenIs$. 

The use of the symbol $\Heisen$ and of the term ``$q$-deformed Heisenberg algebra'' to both mean $\HeisenIs$ is based on~\cite{Hel00}. An excellent discussion may be found in~\cite[pp. 5--11]{Hel00} about which among the 410 references the authors cite, published not later than the year 2000, motivated the algebra $\HeisenIs$, and how it has been consistently denoted by $\Heisen$ or referred to as the $q$-deformed Heisenberg algebra. 

Some works published after~\cite{Hel00} continue to refer to $\HeisenIs$ as $\Heisen$ or as the $q$-deformed Heisenberg algebra. Said works come from varied fields of mathematics, such as Ring Theory~\cite{Hel02,Hel05}, Lie algebras~\cite{Can19a,Can22,Can21,Can20a,Can19b}, Mathematical Physics~\cite{Can20b,Lar03}, and algebraic curves~\cite{Dej09}. All these, and most probably many more, refer to $\Heisen=\HeisenIs$ as the $q$-deformed Heisenberg algebra. No confusion should arise with similar symbol and terminology from recent studies like~\cite{Lop17,Lop22a,Lop22b,Lu15a,Lu15b}, which are \emph{not} the subject of this paper.

From this point onward, we assume that $\F=\C$. If $\repHeisen$ is the unital associative subalgebra of $\algBound$ generated by $\crea$ and $\anni$, we immediately find that there exists a~surjective homomorphism $\TheHom:\Heisen\rightarrow\repHeisen$ of algebras such that $A\mapsto \anni$ and $B\mapsto\crea$. In the common realizations of $\crea$ and $\anni$ as operators on $\ell_2$, such as those in~\cite{Ari76,Chu96,Sch94}, it can be shown by routine arguments that the operator $\lbrack\anni,\crea\rbrack^n$ is not the zero operator for any $n\in\N$. This condition, together with an assumption that $q$ is nonzero and is not a~root of unity, is enough to conclude, using~\cite[Theorem 6.7]{Hel05}, that the homomorphism $\TheHom$ is injective. i.e., The representation $\TheHom$ of $\Heisen$ is faithful. Thus, we have:

\begin{proposition}\label{embeddingProp} The subalgebra $\repHeisen$ of $\algBound$ is isomorphic to $\Heisen$.
\end{proposition}

By this isomorphism, the algebraic properties of $\Heisen$ such as bases, gradations, and others, carry over to $\repHeisen$ with $A\mapsto \anni$ and $B\mapsto\crea$. We discuss these properties in the succeeding remarks.

\begin{remark}\label{PBWrem} By~\cite[Corollary~4.5]{Hel05}, the elements
\begin{eqnarray}
\repLieAB^k, \label{typeLie}\\
\repLieAB^k \anni^l,\label{typeA0}\\
\pcrea^{l}\repLieAB^k, \label{typeB0}
\end{eqnarray}
where $k\in\N$, $l\in\Zplus$, form a~basis for $\repHeisen$. 
\end{remark}

Denote by $\repHeisenLie$ the Lie subalgebra of $\repHeisen$ generated by $\crea$ and $\anni$, or equivalently, $\repHeisenLie$ is the commutator algebra for the creation and annihilation operators of the \mbox{$q$-oscillator}. The basis of $\repHeisen$ consisting of the elements~\eqref{typeLie}--\eqref{typeB0} has a~very important role in the solution of Problem~\ref{qProb} as obtained in~\cite{Can19a}. To describe this solution, we divide~\eqref{typeLie}--\eqref{typeB0} into two groups. For the first group, we take $1=\repLieAB^0$ from~\eqref{typeLie}, and from~\eqref{typeA0},\eqref{typeB0}, we consider those in which the exponent $k$ of $\repLieAB$ is zero and at the same time, the exponent $l$ of either $\crea$ or $\anni$ is at least $2$. That is,
\begin{eqnarray}
1,\quad \anni^l,\quad\pcrea^l,\quad\quad\quad (l\in\N\backslash\{0,1\}).\label{nonLiebasis}
\end{eqnarray}
Then the basis elements of $\repHeisen$ in~\eqref{typeLie}--\eqref{typeB0} that are not in~\eqref{nonLiebasis} are
\begin{eqnarray}
\repLieAB^k, & & \quad\quad\quad\quad (k\in\Zplus),\label{ctypeLie}\\
\anni,\quad\repLieAB^k \anni^l, & & \quad\quad\quad\quad (k,l\in\Zplus),\label{ctypeA0}\\
\crea,\quad\pcrea^{l}\repLieAB^k, & & \quad\quad\quad\quad (k,l\in\Zplus).\label{ctypeB0}
\end{eqnarray}

\begin{remark}\label{LieBasisRem} According to~\cite[Theorem~5.8]{Can19a}, the elements~\eqref{ctypeLie}--\eqref{ctypeB0} form a~basis for $\repHeisenLie$. That is, if $\nonLiesub$ is the vector subspace of $\repHeisen$ spanned by~\eqref{nonLiebasis}, then we have the direct sum decomposition
\begin{eqnarray}
\repHeisen = \nonLiesub\oplus\repHeisenLie.
\end{eqnarray}
\end{remark}

In other words, any finite linear combination of~\eqref{nonLiebasis} cannot be expressed as a~finite linear combination of nested commutators in $\crea$, $\anni$, or equivalently not a~\emph{Lie polynomial in $\crea$, $\anni$}. Thus, the answer to Problem~\ref{qProb} is that the commutator algebra $\repHeisenLie$ is the set of all finite linear combinations of~\eqref{ctypeLie}--\eqref{ctypeB0}. The generators $\crea$, $\anni$ of $\repHeisenLie$ are clearly Lie polynomials in $\crea$, $\anni$, but then it is most natural to ask how the other basis elements of $\repHeisenLie$ from~\eqref{ctypeLie}--\eqref{ctypeB0} can be expressed as Lie polynomials in $\crea$, $\anni$. To show this, we shall exhibit some relations from~\cite{Can19a,Can19b} that express such basis elements as Lie polynomials or linear combinations of nested commutators in $\crea$,~$\anni$, but before exhibiting these relations, we first discuss a~convenient notation using the \emph{adjoint map}.

Given a~Lie algebra $L$ and $x \in L$, recall the adjoint map $\ad x : L \to L$ defined by $ y\mapsto \lbrack x,y\rbrack$. By linearity of the Lie bracket in the second argument, $\ad x$ is a~linear map for any $x \in L$, and by the skew-symmetry of the Lie bracket, $(-\ad x)(y) = -[x,y] = [y,x]$. For convenience in expressing nested commutators, we use the concept of the adjoint map in the following manner.
\begin{remark}\label{adRem} Let $x_1, x_2, \ldots,x_k, y_1, y_2, \dots, y_l \in L$. If $k=2$, the nested commutator $[x_1,[x_2,y]]$ can be expressed as the composition of two adjoint maps given by\linebreak $((\ad x_1) \circ (\ad x_2))(y)$. In general, we have
 \begin{eqnarray}
 ((\ad x_1)\circ(\ad x_2)\circ \cdots \circ (\ad x_k))(y) = [x_1,[x_2,[\cdots,[x_k,y]\cdots]]]. \label{eq:9}
 \end{eqnarray}
For the special case $x_1 = x_2 = \cdots = x_k$, we can write the left-hand side of~\eqref{eq:9} as $(\ad x)^k (y)$. Similarly, if $l=2$, the nested commutator of $[[x,y_1],y_2]$ can be expressed as the composition of two adjoint maps given by $((-\ad y_2)\circ(-\ad y_1))(x)$. This can be generalized as
  \begin{eqnarray}
  ((-\ad y_1) \circ (-\ad y_2)\circ \cdots \circ (-\ad y_l))(x) = [[[\cdots[x,y_l],\cdots], y_2], y_1]. \label{eq:10}
  \end{eqnarray}
We also have the similar result that if $y_1 = y_2 = \cdots = y_l$, then we can write the left-hand side of~\eqref{eq:10} as $(-\ad y)^l (x)$.
\end{remark}

\begin{remark}
From~\cite[Section~5]{Can19a}, if $q$ is nonzero and is not a~root of unity, then the basis elements of $\repHeisenLie$ from~\eqref{ctypeLie}--\eqref{ctypeB0} can be expressed as Lie polynomials in $\crea$ and $\anni$ through the relations
\begin{eqnarray}
\repLieAB^{k+2} & = & \frac{-q^k (1-q)}{1-q^{k+1}}\sum_{i=0}^k \frac{\left(\left(\ad \crea\right)\circ\left(-\ad \repLieAB\right)^{k}\circ(\ad\anni)\right)\left(\lbrack \anni,\crea\rbrack\right)}{(q-1)^{1+i}},\label{combase1}\\
\repLieAB^{k+1}\anni^l & = & -\frac{\left(\left(-\ad \repLieAB\right)^k \circ\left(-\ad \anni\right)^{l+1}\right)\left(\crea\right)}{(1-q)^l (q^l -1)^k}, \label{combase2}\\
\pcrea^l \repLieAB^{k+1} & = & \frac{\left(\left(\ad \crea\right)^{l-1}\circ\left(\ad \repLieAB\right)^{k+1}\right)\left(\crea\right)}{(q-1)^{k+1}(1-q^{k+1})^{l-1}},\label{combase4}
\end{eqnarray}
which hold for any $k\in\N$, and any $l\in\Zplus$. The operators $\crea$, $\anni$, and $\repLieAB$ are clearly Lie polynomials in $\crea$, $\anni$, and so we need not consider them in the relations~\eqref{combase1}--\eqref{combase4}. The commutator table for these basis elements of $\repHeisenLie$ can be gleaned from~\cite[Table~1]{Can19a}, and a~detailed discussion of the computational aspects of these relations can be found in~\cite[Section~2.1]{Can19b}.
\end{remark}

If $q$ is a~root of unity other than $1$, there are special cases for the constructions in~\mbox{\eqref{combase1}--\eqref{combase4}}, and the basis~\eqref{ctypeLie}--\eqref{ctypeB0} of $\repHeisenLie$ is reduced according to some conditions in the exponents. These can be found in~\cite[Section~4.1]{Can19b}. If $q=0$, an entirely different basis and nested commutator constructions can be found in~\cite[Section~4]{Can19a}. These two cases, however, are not covered by the representation theorem~\cite[Theorem 6.7]{Hel05}, and so the possibly nonfaithful representation $\repHeisen$ of $\Heisen$ for these two cases may lead to further restrictions in the algebraic structure of $\repHeisenLie$. We reserve these two cases as for a~possibly interesting continuation of the results presented in this paper, and so we have the following.

\begin{assumption}\label{nontorsionAs}
The scalar $q$ is nonzero and is not a~root of unity.
\end{assumption}

\subsection{The objective of this study: extending the commutator algebra for \texorpdfstring{$\crea$, $\anni$}{a+, a}}\label{objSec}

As a~natural continuation of Problem~\ref{qProb}, we consider extending the generating set of the operator Lie algebra by another important operator from the $q$-oscillator. One possibility is the \emph{number operator $\Num$} defined by the action $\Num\lket n\rket := n\lket n\rket$ for any $n\in\N$. The canonical commutation relations that relate $\Num$ to $\crea$ and $\anni$ are
\begin{eqnarray}
\lbrak\Num,\crea\rbrak=\crea,\quad\lbrak\anni,\Num\rbrak =\anni,\label{Ncommrel}
\end{eqnarray}
which hold even for the $q$-oscillator. We note the simple nature of the relations~\eqref{Ncommrel}, in which we can see that the result of performing the commutator operation on any of \mbox{$\crea$, $\anni$} involving $\Num$ simply results to linear combinations of $\crea$, $\anni$. By an easy argument, an immediate consequence is that the extended commutator algebra for $\crea$, $\anni$, $\Num$ is simply the direct sum $\C\Num\oplus\repHeisenLie$, where $\C\Num$ is the set of all complex scalar multiples of $\Num$. As will be demonstrated in the results in this paper, it is the operator $q^{\kappa\Num}$, defined by the action
\begin{eqnarray}
q^{\kappa\Num}\lket n\rket:=q^{\kappa n}\lket n\rket,\label{kparam}
\end{eqnarray}
for some scalar $\kappa$, that gives a~richer algebraic structure in extending the commutator algebra for $\crea$, $\anni$. By some choice of $\kappa$, our results for the extended commutator algebra for $\crea$, $\anni$, $q^{\kappa\Num}$ turn out to be intimately related to a~$q$-oscillator realization of an \emph{Askey-Wilson algebra}. The three-parameter Askey-Wilson algebra $\AW$ was introduced in~\cite{Zhe91}, and is said to be important by itself as a~dynamical symmetry algebra in problems in which the Askey-Wilson polynomials arise as eigenfunctions~\cite[p.~1147]{Zhe91}. The algebra $\AW$ and its extension to generalized families of similar algebras have been extensively studied in algebraic combinatorics, among many other fields. See, for instance, the comprehensive discussion in~\cite[Section~1]{Ter11} about how the Askey-Wilson algebra and its variants are studied in algebraic combinatorics, integrable systems, and quantum mechanics, or the discussion in~\cite[Section~1]{Hua15} about how $\AW$ is related to several significant quantum groups. The algebraic perspective in this work, however, differs from that in~\cite{Hua15} or~\cite{Ter11}. We do not study ring-theoretic generalizations or finite-dimensional representation theory, but the Lie structure of associative algebras, such as the perspective in the studies~\cite{Can15,Can19a,Can19b}. We show that our results about the extended commutator algebra for the $q$-oscillator shed light on some related Lie structure in a~$q$-oscillator representation of the algebra $\AW$.

\section{Preliminaries: Algebraic structure of the \texorpdfstring{$q$}{q}-deformed Heisenberg algebra \texorpdfstring{$\Heisen$}{H(q)}}\label{prelSec}

We review some important results about the algebra $\Heisen$, which has a~presentation with generators $A,B$ and relation $AB-qBA=1$. Let $C:=\lbrack A,B\rbrack$. In terms of $A,B,C$, the basis~\eqref{typeLie}-\eqref{typeB0} of $\repHeisen=\Heisen\subseteq\Proto$ can be rewritten as
\begin{eqnarray}
C^k,\quad C^k A^l, \quad C^k B^l,\quad\quad (k\in\N,\ l\in\Zplus).\label{ABCbasis}
\end{eqnarray}
For each $l\in\N$, define $\gradsub_{-l}$ as the span of all elements $C^k A^l$ ($k\in\N$), and $\gradsub_{l}$ as the span of $C^k B^l$ ($k\in\N$).

\begin{remark}\label{CBformRem} In the basis~\eqref{ABCbasis} of $\Heisen$, we use basis elements of the form $C^k B^l$ for the product of powers of $C$ and $B$, instead of $B^l C^k$ which was used in~\cite{Can19b,Hel05}. This may be justified by the reordering formula $B^l C^k =q^{-kl}C^k B^l$ (for any $k,l\in\N$) which is from \mbox{\cite[equation 18]{Hel05}}. This change in normal form for this particular type of basis element is valid because of the assumption that $q$ is nonzero (Assumption~\ref{nontorsionAs}).
\end{remark}

Denote the set of all integers by $\Z$. For any vector subspaces $\mathfrak{A}$ and $\mathfrak{B}$ of any algebra $\mathcal{A}$, we define $\mathfrak{A}\mathfrak{B}$ as the span of all elements $ab$ where $a\in\mathfrak{A}$ and $b\in\mathfrak{B}$. Following~\mbox{\cite[Proposition~2.4, Corollary~4.5]{Hel05}}, the collection $\{\gradsub_k\ :\ k\in\Z\}$ is a~$\Z$-gradation of $\Heisen$. That is, we have the direct sum decomposition
\begin{eqnarray}
\Heisen = \bigoplus_{k\in\Z}\gradsub_k,
\end{eqnarray}
and that for any $h,k\in\Z$,
\begin{eqnarray}
\gradsub_h\gradsub_k\subseteq\gradsub_{h+k}.
\end{eqnarray}
Thus, for instance, if $m,k\in\N$ and $n,l\in\Zplus$ with $n>l$, then $C^m B^n \cdot C^k A^l \in\gradsub_{n-l}$. However, this does not give us any information about how to express $C^m B^n \cdot C^k A^l$ as a~linear combination of the basis elements $C^h B^{n-l}$ ($h\in\N$) of $\gradsub_{n-l}$. In general, we are interested in the \emph{structure constants} of $\Heisen$ with respect to the basis~\eqref{ABCbasis}. We discuss these in the following subsections.

\subsection{Structure constants of \texorpdfstring{$\Heisen$}{H(q)} as an associative algebra}\label{assoStructSubsec}

In this subsection, we recall the reordering formula and structure constants of $\Heisen$ as an associative algebra, which is studied in~\cite{Can19a,Can19b,Hel00,Hel05}. We present here a~simplified discussion of~\cite[Section 2.1]{Can19b}. Define $\{0\}_q:=0$, and for each $n\in\Zplus$, we recursively define $\{n\}_q:=1+q\{n-1\}_q$. That is, $\{n\}_q=1+q+\cdots+q^{n-1}$. If $q\neq 1$, then $\{n\}_q=\frac{1-q^n}{1-q}$. The \emph{Gaussian binomial coefficients} or \emph{$q$-binomial coefficients} are recursively defined by
\begin{eqnarray}
{{n}\choose{0}}_q & = & 1,\\
{{0}\choose{k+1}}_q & = & 0,\\
{{n+1}\choose{k+1}}_q & = & {{n}\choose{k}}_q + q^{k+1}{{n}\choose{k+1}}_q,\label{qbirec}
\end{eqnarray}
for any $n,k\in\N$. These $q$-binomial coefficients satisy the symmetry property ${{n}\choose{k}}_q={{n}\choose{n-k}}_q$ for any $k\in\{0,1,\ldots,n\}$, and as a~consequence also the properties
\begin{eqnarray}
{{n}\choose{0}}_q = {{n}\choose{n}}_q & = & 1,\label{qbi1}\\
{{n}\choose{1}}_q = {{n}\choose{n-1}}_q & = & \{n\}_q.\label{qbi2}
\end{eqnarray}
For each $l\in\Zplus$ and each $i\in\{0,1,\ldots,l\}$ we define
\begin{eqnarray}
 c_i(l) := (q-1)^{-l}(-1)^{l-i}q^{{i+1}\choose{2}}{{l}\choose{i}}_q.\label{cidef}
\end{eqnarray}
Let $k,m\in\N$, and let $l,n\in\Zplus$. As derived in~\cite[Section 2.1]{Can19b}, the structure constants of the basis~\eqref{ABCbasis}, given $l,n\in\Zplus$, may be determined from the relations

\begin{eqnarray}
C^m A^n \cdot C^k A^l & = & q^{kn}C^{m+k}A^{n+l},\label{medA2}\\
C^m B^n \cdot C^k B^l & = & q^{-kn}C^{m+k}B^{n+l},\label{medB2}\\
C^m A^n \cdot C^k B^l & = & \sum_{i=0}^{\min\{l,n\}}q^{(i+k)n-il}c_i(l)\cdot C^{m+i+k}A^{n-l},\quad\qquad\ \ \ (n\geq l),\label{ABnbigger}\\
 & = & \sum_{i=0}^{\min\{l,n\}}q^{kn}c_i(n)\cdot C^{m+i+k}B^{l-n},\quad\quad\quad\qquad\ \ \ \ (l>n),\label{ABnsmaller}\\
C^m B^n \cdot C^k A^l & = & \sum_{i=0}^{\min\{l,n\}}q^{-(i+k)n}c_i(l)\cdot B^{n-l}C^{m+k+i},\quad\quad\qquad\ (n\geq l)\label{BAnbigger},\\
 & = & \sum_{i=0}^{\min\{l,n\}}q^{-(i+k)n}c_i(n)\cdot C^{i+m+k}A^{l-n},\quad\quad\quad\ \ \ \ (l>n).\label{BAnsmaller}
\end{eqnarray}
There is some simplification in the scalar coefficients regarding $c_i(l)$, since in~\cite{Can19b} some other scalar coefficients were defined but are actually dependent on $c_i(l)$. We summarize in Table~\ref{table:structconst} below how the relations~\eqref{medA2} to~\eqref{BAnsmaller} give the structure constants of $\Heisen$ with respect to the basis~\eqref{ABCbasis}.
\begin{center}
\begin{table}[h]
\centering
\begin{tabular}
{|c|c|c|c|}
\hline
$\cdot$ & $C^k A^l$ & $C^k B^l$ \\
\hline
$C^m A^n$ &~\eqref{medA2} &~\eqref{ABnbigger},~\eqref{ABnsmaller} \\
\hline
$C^m B^n$ &~\eqref{BAnbigger},~\eqref{BAnsmaller} &~\eqref{medB2} \\
\hline
\end{tabular}
\caption{\label{table:structconst}Relations that can be used to obtain the structure constants of the associative algebra $\Heisen$ with respect to the basis~\eqref{ABCbasis}}
\end{table}
\end{center}

\subsection{Structure constants of \texorpdfstring{$\Heisen$}{H(q)} as a~Lie algebra}

We extend the information from Section~\ref{assoStructSubsec} by determining the structure constants of $\Heisen$ as a~Lie algebra, still with respect to the basis~\eqref{ABCbasis}. Let $m,r\in\N$, and let $n,s\in\N\backslash\{0\}$. Using the computational techniques in the previous subsection, it is routine to show that the relations
\def\colIII{\LieAB^r}
\def\colIV{\LieAB^r A^s}
\def\colV{B^s \LieAB^r}

\def\rowIII{\LieAB^m}
\def\rowIV{\LieAB^m A^n}
\def\rowV{B^n \LieAB^m}

\begin{eqnarray}
\lbrack C^m A^n,C^r B^s \rbrack & = & \sum_{i=0}^{\min\{r,s\}} q^{rs}\left(1-q^{E(i)}\right)c_i(n)C^{m+i+r},\qquad\qquad\qquad\quad (n=s),\label{equalns}\\
& = & \sum_{i=0}^{\min\{r,s\}}q^{nr+(s-n)(n-i)}\left(1-q^{F(i)}\right)c_i(n)C^{m+i+r}B^{s-n},\quad (n<s),\label{nless}\\
& = & \sum_{i=0}^{\min\{r,s\}}q^{nr+i(n-s)}\left(1-q^{G(i)}\right)c_i(s)C^{m+i+r}A^{n-s},\qquad\ \ \ (n>s)\label{nbig},
\end{eqnarray}
where
\begin{eqnarray}
E(i) & := & n(m+r-i),\nonumber\\
F(i) & := & n(2m+r+n-s)-in-ms, \label{Fidef}\nonumber\\
G(i) & := & s(m+2r),\nonumber
\end{eqnarray}
hold in $\Heisen$ where $i\in\{1,2,\ldots,\min\{n,s\}\}$. Using~\eqref{medA2},~\eqref{medB2}, we also have
\begin{eqnarray}
\lbrack C^m A^n,C^r A^s \rbrack & = & q^{nr}(1-q^{ms-nr})C^{m+r}A^{n+s},\label{computesameA}\\
\lbrack C^m B^n,C^r B^s \rbrack & = & q^{-nr}(1-q^{nr-ms})C^{m+r}B^{n+s}.\label{computesameB}
\end{eqnarray}

Similar to that done in Table~\ref{table:structconst}, we summarize the structure constants of $\Heisen$ as a~Lie algebra with respect to the basis~\eqref{ABCbasis} in Table~\ref{table:Liestructconst} below. Because of the skew-symmetry of the Lie bracket, the relations~\eqref{equalns} to~\eqref{computesameB} give us all the structure constants of $\Heisen$ as a~Lie algebra.
\begin{center}
\begin{table}[h]
\centering
\begin{tabular}
{|c|c|c|c|}
\hline
$\lbrack -,-\rbrack$ & $C^r A^s$ & $C^r B^s$ \\
\hline
$C^m A^n$ &~\eqref{computesameA} &~\eqref{equalns},~\eqref{nless},~\eqref{nbig} \\
\hline
$C^m B^n$ & - &~\eqref{computesameB} \\
\hline
\end{tabular}
\caption{Relations that can be used to obtain the structure constants of $\Heisen$ as a~Lie algebra with respect to the basis~\eqref{ABCbasis}}\label{table:Liestructconst}
\end{table}
\end{center}

\section{An extended operator algebra for the \texorpdfstring{$q$}{q}-oscillator}

Let $\freq$ be a~positive real number. Earlier, the parameter $q$ was chosen to be any nonzero complex number that is not a~root of unity. From this point onward, we further narrow down the choice of $q$ by setting $q=e^{-2\freq}$. Following~\cite[equations (6.5)--(6.6)]{Zhe91}, we have the relations
\begin{eqnarray}
e^\freq\anni\crea-e^{-\freq}\crea\anni & = & e^\freq,\label{Hqrel}\\
\lbrak\anni,\crea\rbrak & = & e^{-2\freq\Num},\label{commrel}
\end{eqnarray}
for the $q$-oscillator. Define $\genzero$ as the operator $e^{\freq\Num}$, i.e., $\genzero\lket n\rket :=e^{\freq n}\lket n\rket$ for any $n\in\N$. With reference to the discussion in Section~\ref{objSec} about~\eqref{kparam}, we are here setting $\kappa=\frac{\freq}{\ln q}$. Denote by $\Protoq$ the algebra (of operators on $\ell_2$) generated by $\genzero$, $\anni$, $\crea$. As a~consequence of~\eqref{Hqrel},\eqref{commrel}, it can be shown by routine calculations that the relations
\begin{eqnarray}
\anni\crea & = & \frac{1-e^{-2\freq}\oscLie}{1-e^{-2\freq}},\label{protoq1}\\
\crea\anni & = & \frac{1-\oscLie}{1-e^{-2\freq}},\\
\anni\oscLie & = & e^{-2\freq}\oscLie\anni,\\
\crea\oscLie & = & e^{2\freq}\oscLie\crea,\\
\anni\genzero & = & e^{\freq}\genzero\anni,\\
\crea\genzero & = & e^{-\freq}\genzero\crea,\\
\oscLie\genzero & = & \genzero\oscLie,\\
\genzero^2 \oscLie & = & \algI,\label{protoq8}
\end{eqnarray}
hold in $\Protoq$. We are interested in the question of whether the relations~\eqref{protoq1}--\eqref{protoq8} are enough to define a~presentation for $\Protoq$. In the interest of mathematical rigour, this is no trivial matter, and so in this section we prove that such a~presentation can indeed be obtained for $\Protoq$.
\begin{definition}
Given a~nonzero real number $\omega$, if $q=e^{-2\omega}$, then we define $\Proto=\Proto(q)$ as the unital associative algebra over $\C$ with generators $K,A,B,C$ subject to the relations
\begin{eqnarray}
AB & = & \frac{1-qC}{1-q},\label{proto1}\\
BA & = & \frac{1-C}{1-q},\label{proto2}\\
AC & = & qCA,\label{proto3}\\
BC & = & q^{-1}CB,\label{proto4}\\
AK & = & q^{-\frac{1}{2}}KA,\label{proto5}\\
BK & = & q^{\frac{1}{2}}KB,\label{proto6}\\
CK & = & KC,\label{proto7}\\
K^2 C & = & \algI.\label{proto8}
\end{eqnarray}
\end{definition}
By the defining relations~\eqref{proto1} to~\eqref{proto8} of the algebra $\Proto$, and using the Diamond Lemma for Ring Theory~\cite[Theorem 1.2]{Ber78}, the elements

\begin{eqnarray}
K^h C^k,\quad K^h C^k B^l,\quad K^h C^k A^l,\quad\quad\quad\quad (h,k\in\N,\ l\in\Zplus),\label{protoBasis}
\end{eqnarray}
with the condition
\begin{eqnarray}
k\in\Zplus \quad \Rightarrow \quad h\in\{0,1\},\label{protoBasisCond}
\end{eqnarray}
form a~basis for $\Proto$.

Comparing~\eqref{protoq1}-\eqref{protoq8} with~\eqref{proto1}-\eqref{proto8}, if we consider the assignment
\begin{eqnarray}
K\mapsto\genzero,\qquad A\mapsto\anni,\qquad B\mapsto\crea,\qquad C\mapsto\oscLie,\label{genassgn}
\end{eqnarray}
we find that the generators of $\Protoq$ satisfy the defining relations of $\Proto$. Thus, we say that $\Protoq$, a~concrete algebra of operators on $\ell_2$, is a~representation of the abstractly defined algebra $\Proto$. More precisely, there exists a~homorphism $\Psi_0:\Proto\rightarrow\Protoq$ of algebras that performs~\eqref{genassgn}. Our goal, at this point, is to prove the faithfulness of the representation $\Protoq$ of $\Proto$, or equivalently, the injectiveness of $\Psi_0$.

\subsection{Deformed commutator maps and the faithfulness of the representation \texorpdfstring{$\Protoq$}{Pq}}
Given an algebra $\mathcal{A}$ over $\C$, we recall, from~\cite{Hel05}, the \emph{deformed commutator mappings}
\begin{eqnarray}
\dcmTH_{\alpha,t} & : & \gamma\mapsto\alpha\gamma-t\gamma\alpha,\nonumber\\
\dcmH_{\beta,t} & : & \gamma\mapsto\gamma\beta-t\beta\gamma,\nonumber
\end{eqnarray}
given $\alpha,\beta\in\mathcal{A}$ and $t\in\C$. Consequently, any (two-sided) ideal of $\mathcal{A}$ is invariant under a~deformed commutator mapping. We shall be concerned with such deformed commutator mappings on $\Proto$ of very specific types. Given $n\in\Q$, succeeding proofs and computations shall involve the deformed commutator mappings
\begin{eqnarray}
\dcmTH_n & := & \dcmTH_{A,q^n}:\gamma\mapsto A\gamma-q^n \gamma A,\nonumber\\
\dcmH_n & := & \dcmH_{B,q^n}:\gamma\mapsto \gamma B-q^n B\gamma.\nonumber
\end{eqnarray}
We collect all such maps in $\ALLdcm:=\{\dcmTH_n,\ \dcmH_n\ :\ n\in\Q\}$. Some subsets of $\ALLdcm$ shall be of importance in later proofs, such as $\Zdcm:=\{\dcmTH_n,\ \dcmH_n\in\ALLdcm\ :\ n\in\Z\}$, and the set difference $\nonZdcm=\{\dcmTH_n,\ \dcmH_n\in\ALLdcm\ :\ n\notin\Z\}$. Using the computational techniques in Section~\ref{assoStructSubsec}, the identities
\begin{eqnarray}
\dcmTH_n(C^k) & = & q^{k}(1-q^{n-k})C^k A,\label{vanishTH}\\
\dcmTH_n(C^k A^l) & = & q^{k}(1-q^{n-k})C^k A^{l+1},\\
\dcmTH_n(C^k B^l) & = & q^{k}\frac{1-q^{n-k}}{1-q}C^k B^{l-1}-q^{k+1}\frac{1-q^{n-k-l}}{1-q}C^{k+1}B^{l-1},\label{lowerBl}\\
\dcmH_n(C^k) & = & q^{-k}(1-q^{n+k})C^k B,\label{vanishH}\\
\dcmH_n(C^k B^l) & = & q^{-k}(1-q^{n+k})C^k B^{l+1},\label{raiseBl}\\
\dcmH_n(C^k A^l) & = & q^{-k}\frac{1-q^{n+k}}{1-q}C^k A^{l-1}-q^{-k}\frac{1-q^{n+k+l}}{1-q}C^{k+1}A^{l-1},\label{lowerAl}
\end{eqnarray}
hold in $\Heisen$ for any $k\in\N$, and $n\in\Q$ and any $l\in\Zplus$. Setting $n=k$ in~\eqref{vanishTH} and $n=-k$ in~\eqref{vanishH}, we have
\begin{eqnarray}
\dcmTH_k(C^k) = 0 = \dcmH_{-k}(C^k),\qquad (k\in\N).\label{dcmzero}
\end{eqnarray}
Setting $n=k+1$, $l=1$ in~\eqref{lowerBl} and $n=-k$, $l=1$ in~\eqref{lowerAl}, we have
\begin{eqnarray}
\dcmTH_{k+1}(C^k B) & = & q^k C^k,\label{CBtoC}\\
\dcmH_{-k}(C^k A) & = & -q^{-k}C^{k+1}.\label{CAtoC}
\end{eqnarray}
In order to exhibit some more interesting properties of the deformed commutator map $\dcmTH_n$ that shall be useful in later computations and proofs, we first recall the notion of a~$q$-derivative. It is the mapping
\begin{eqnarray}
D_q : f(x)\mapsto \frac{f(x)-f(qx)}{x-qx},\nonumber
\end{eqnarray}
where, in this work, we take $f$ to be any element of the polynomial algebra $\C[x]$ on the indeterminate $x$. Throughout, we use the notation
\begin{eqnarray}
\{i\}_q! := \prod_{j=1}^i \{j\}_q.\nonumber
\end{eqnarray}
Using~\cite[Lemma C.3]{Hel00} and induction, it is routine to show that
\begin{eqnarray}
\left(D_q^i \right)(x^n)=\{i\}_q!{{n}\choose{i}}_qx^{n-i},\label{higherDq}
\end{eqnarray}
for any $i,n\in\N$. Moreover, for each $k\in\N$, we have
\begin{eqnarray}
\dcmTH_k(C^k B^l)=q^{k-l+1}C^{k+1}\cdot D_q(B^l).\label{thetader}
\end{eqnarray}
We note here that~\eqref{thetader} is an adjusted form of a~related formula from~\cite[Lemma 5.7]{Hel05}. As explained in Remark~\ref{CBformRem}, the adjustment is due to our use of basis elements of the form $C^k B^l$ as normal form instead of using the elements of the form $B^l C^k$ which were used in~\cite{Hel05}. It is possible to apply the mapping $\dcmTH_k$ on both sides of~\eqref{thetader} and use the identity~\eqref{higherDq} to simplify the resulting right-hand side. Continuing with this manner of computation in a way that the number $n\in\N$ of times the mapping $\dcmTH_k$ is applied to $C^k B^l$ does not exceed $l$, and by some routine computations and induction, we have
\begin{eqnarray}
\dcmTH_k^n (C^k B^l)=q^{n(k-l+n)}\{n\}_q!{{l}\choose{n}}_qC^{k+n}B^{l-n},\qquad (k,n\in\N,\ l\in\Zplus,\ l\leq n.)\label{powerofdcm}
\end{eqnarray}

The identity~\eqref{powerofdcm} suggests some computational significance of deformed commutator mappings in the sense that applying a~deformed commutator mapping repeatedly reduces a~basis element of $\Heisen$ into a~power of $C$. In an interesting flow of computations and proofs, we show that this can be generalized\textemdash any nonzero element of $\Heisen$, after an application of a~suitable number of particular types of deformed commutator mappings, reduces to a~power of $C$.

We start with a~few properties in the following proposition and develop our main algorithms in the three lemmas that follow. These are actually our reformulation of~\cite[Lemma~5.7-6.4]{Hel05}. In fact, we give a~reformulation, an exposition, and a~generalization all at the same time. Our techniques as presented here involve more explicit reference to computations and structure constants, as contrasted by the ring-theoretic approach in~\cite{Hel05}.

\begin{lemma}\label{algoBLem} Let $b\in\Zplus$. For each nonzero $U\in\bigoplus_{i=1}^b \gradsub_i$, there exists a~composition $\theta$ of elements of $\Zdcm$, and some nonzero $V\in\gradsub_0=\C[C]$ such that $\theta(U)=V$.
\end{lemma}
\begin{proof}
We use induction on $b$. Suppose $b=1$. Let $U$ be a~nonzero element of $\gradsub_1$. Then there exists a~nonzero $\beta\in\C[C]$ such that $U=\beta B$. Let $k$ be the polynomial degree of $\beta$. Choose any positive integer $n>k+1$. If $c_k$ is the leading coefficient of the polynomial $\beta$, then by~\eqref{lowerBl}, $\dcmTH_n(\beta B)=\beta'$ for some $\beta'\in\C[C]$ with polynomial degree $k+1$, and with leading coefficient $-q^{k+1}\{n-k-1\}_qc_k=-q^{k+1}\frac{1-q^{n-(k+1)}}{1-q}c_k$, which is nonzero by Assumption~\ref{nontorsionAs} and because $n>k+1$.

Suppose that for all positive integers $c<b$, any nonzero element of $\bigoplus_{i=1}^c \gradsub_i$ satisfies the statement. Let $U$ be a~nonzero element of $\bigoplus_{i=1}^b \gradsub_i$. Then there exists $U'\in\bigoplus_{i=1}^{b-1}\gradsub_i$ and some $\beta\in\C[C]$ such that $U=U'+\beta B^{b}$. If $\beta=0$, then we are done, and so we further assume $\beta\neq 0$. Let $k$ be the polynomial degree of $\beta$. By computations similar to those described for the case $b=1$, we have $\dcmTH_n(U)=\dcmTH_n(U')+\beta'B^{b-1}$ for some $\beta'\in\C[C]$ with polynomial degree $k+1$, and with leading coefficient $-q^{k+1}\{n-k-b\}_qc_k\neq 0$. Whenever $b\in\N\backslash\{0,1\}$, the element $\dcmTH_n(U')$ is in $\bigoplus_{i=1}^{b-2}\gradsub_i$ by~\eqref{lowerBl}, and so no term in $\dcmTH_n(U')$ can cancel the term $-q^{k+1}\{n-k-b\}_qc_kC^{k+1}B^{b-1}$ in $\beta'B^{b-1}$. Thus, $0\neq\dcmTH_n(U)\in\bigoplus_{i=1}^{b-1}\gradsub_i$, and the element $\dcmTH_n(U)$ satisfies the inductive hypothesis. That is, there exists a~composition $\phi$ of elements of $\Zdcm$, and some nonzero $V\in\gradsub_0=\C[C]$ such that $\phi(\dcmTH_n(U))=V$. Take $\theta=\phi\circ\dcmTH_n$. This completes the induction.
\end{proof}

\begin{lemma}\label{algoCLem} For any nonzero $U\in\C[C]=\gradsub_0$, there exists a~nonzero scalar $a$, some positive integer $N$ and a~composition $\phi$ of elements of $\Zdcm$ such that $\phi(U)=aC^N$.
\end{lemma}
\begin{proof}
For each nonzero $U\in\C[C]$, there exist numbers $k,n\in\N$ such that for some scalars $c_k$, $c_{k+1}$, $\ldots$, $c_{k+n}$, with $c_k$ and $c_{k+n}$ nonzero, we have
\begin{eqnarray}
U=c_kC^k + c_{k+1}C^{k+1}+\cdots+c_{k+n}C^{k+n}.\label{rangeU}
\end{eqnarray}
That is, $n$ is the difference between the highest exponent of $C$ and the lowest exponent of $C$ that appear with nonzero scalar coefficient in $U$. We use induction on $n$. First, we consider the case $n=0$. If $k=0$, then take $\phi=\dcmH_0\circ\dcmTH_1$, $a=q(q-1)c_0$ and $N=1$, while if $k\neq 0$, then we take $\phi$ as the empty composition or the identity map, $a$ as $c_k$ and $N$ as $k$, and we are done. Suppose that for some $n$, the statement holds for all nonzero elements of $\C[C]$ in the linear combination of which the difference $m$ between the highest exponent of $C$ and the lowest exponent is such that $m<n$. By routine computations that involve~\eqref{lowerBl},~\eqref{dcmzero},~\eqref{CBtoC}, there exist nonzero scalars $e_i$ with $i\in\{k+1,\ldots, k+n\}$ such that
\begin{eqnarray}
\left(\dcmTH_{k+n+1}\circ\dcmH_{-k}\right)(U) &=& c_{k+1}e_{k+1}C^{k+1}+\cdots+c_{k+n-1}e_{k+n-1}C^{k+n-1}\nonumber\\
&&+(c_{k+n-1}e_{k+n}+c_{k+n}(1-q^n))C^{k+n}.\label{phewC}
\end{eqnarray}
The number of nonzero terms in the right-hand side of~\eqref{phewC} is at most $k+n-(k+1)=n-1$. Thus, the inductive hypothesis applies to $V:=\left(\dcmTH_{k+n+1}\circ\dcmH_{-k}\right)(U)$, and so there exists a~nonzero scalar $a$, some positive integer $N$ and some composition $\phi_0$ of elements of $\Zdcm$ such that $\phi_0(V)=aC^N$, and so
\begin{eqnarray}
\left(\phi_0\circ\dcmTH_{k+n+1}\circ\dcmH_{-k}\right)(U) = \phi_0(V)=aC^N.\nonumber
\end{eqnarray}
Take $\phi=\phi_0\circ\dcmTH_{k+n+1}\circ\dcmH_{-k}$, and this completes the induction.
\end{proof}

Our Lemmas~\ref{algoBLem} and \ref{algoCLem} now culminate into the following result, which is our generalization of the rule~\eqref{powerofdcm}, and also our reformulation and generalization of~\cite[Theorem 6.4]{Hel05}.

\begin{corollary}\label{magicCor} For any nonzero $U\in\Heisen$, there exists a~composition $\psi$ of elements of $\Zdcm$, some nonzero scalar $a$ and some positive integer $N$ such that $\psi(U)=aC^N$.
\end{corollary}
\begin{proof}
Any nonzero $U\in\Heisen$ is a~finite linear combination of the elements described in~\eqref{ABCbasis}, and so there exist $a,b\in\N$ such that $U\in\bigoplus_{i=-a}^b \gradsub_i$. By~\eqref{vanishH}--\eqref{lowerAl}, there exists a~composition $\eta$ of elements of $\Zdcm$ such that $\eta(U)$ satisfies the hypotheses of Lemma~\ref{algoBLem}, by which there exists a~composition $\theta$ of elements of $\Zdcm$ such that $(\theta\circ\eta)(U)$ satisfies the hypotheses of Lemma~\ref{algoCLem}. Consequently, by Lemma~\ref{algoCLem}, there exists a~composition $\phi$ of elements of $\Zdcm$ such that if $\psi:=\phi\circ\theta\circ\eta$, and thus there exists a~nonzero scalar $a$ and some positive integer $N$ such that $\psi(U)=aC^N$.
\end{proof}

As the reader may notice, our proof of Corollary~\ref{magicCor} and our proof of the supporting lemmas constitute a~logical flow that differs from the formulation in~\mbox{\cite[Lemma~5.7-6.4]{Hel05}}. First, we deal with the commutator mappings and basis elements, and eventually arbitrary nonzero elements, of $\Heisen$ without reference to any two-sided ideal. The proofs do not rely heavily on ring-theoretic notions and arguments such as divisibility in a~ring. Our proofs are based on the structure constants and computational techniques as described in Section~\ref{assoStructSubsec}, which is based on the deeper discussion in~\cite{Can19b}. Our only dilemma so far is that, since we are extending this reduction process into the bigger algebra $\Proto$ in which $\Heisen$ is embedded, we have to take into account the effect of $K$ in the computations involving deformed commutator mappings. Thus, we have the following.

\begin{proposition}\label{KdcmProp} For any $n\in\Q$ and any $U\in\Heisen$, we have
\begin{eqnarray}
\dcmTH_n(KU) & = & q^{-\frac{1}{2}}K\dcmTH_{n+\frac{1}{2}}(U),\label{KdcmTH}\\
\dcmH_n(KU) & = & q^{\frac{1}{2}}K\dcmH_{n-\frac{1}{2}}(U).\label{KdcmH}
\end{eqnarray}
\end{proposition}
\begin{proof}
By the linearity of deformed commutator mappings, letting $U$ be a~basis element of $\Heisen$ among those listed in~\eqref{ABCbasis} will suffice. If $U=C^k$ for some $k\in\N$, we have that $\dcmTH_n(K^k)=AKC^k -q^n KC^k A$, which by~\eqref{proto5}, becomes 
    \[\dcmTH_n(K^k)=q^{-\frac{1}{2}}K(AC^k -q^{n+\frac{1}{2}}C^k A)=q^{-\frac{1}{2}}K\dcmTH_{n+\frac{1}{2}}(C^k).\]
    The other cases for~\eqref{KdcmTH}, which are $U=C^k A^l$ and the case $U=C^k B^l$ for some $k,l\in\N$ with $l\in\Zplus$, are proven similarly. The proof for the identity~\eqref{KdcmH} involves similar routine computations.
\end{proof}

\begin{proposition}\label{nonzeroQdcmProp} If $\psi$ is (finite) composition of elements of $\nonZdcm$, then for any nonzero $U\in\Heisen$,
\begin{eqnarray}
0\neq \psi(U)\in\Heisen.\nonumber
\end{eqnarray}
\end{proposition}
\begin{proof}
We use induction on the number $t$ of deformed commutator mappings in the composition $\psi$. Suppose $t=1$, that is, for some non-integer $m\in\Q$, either $\psi=\dcmH_m$ or $\psi=\dcmTH_m$. Consider the case $\psi=\dcmH_m$. Let $U$ be any nonzero element of $\Heisen$, and let $a,b\in\Z$ such that $U=V+W$ where $V\in\bigoplus_{i=a}^{b-1}\gradsub_i$ and $W\in\gradsub_b$. Without loss of generality, we assume $W\neq 0$. That is, $b$ is the maximum index of a~$\Z$-gradation subspace to which some terms in $U$ belong. Consequently, there is a~nonzero polynomial $\gamma\in\C[C]$ such that $W=\gamma B^b$ if $b\in\N$, or $W=\gamma A^{-b}$ if $-b\in\Zplus$. In either case, according to~\eqref{raiseBl},~\eqref{lowerAl}, $\dcmH_m(W)\in\gradsub_{b+1}$.

Let $k$ be the polynomial degree of $\gamma$ and let $c_k$ be the leading coefficient. Using the identities\eqref{raiseBl},~\eqref{lowerAl}, we find that there exists a~polynomial $\gamma'\in\C[C]$ of degree either $k$ or $k+1$, with leading coefficient $d=\frac{q^k}{(1-q)^i}(1-q^{m+T})c_k$ (where $k$ and $T$ are integers, and $i$ is either $0$ or $1$) such that $\dcmH_m(W)=\gamma' B^{b+1}$ if $b\in\N$, or $\dcmH_m(W)=\gamma' A^{-b+1}$ if $-b\in\Zplus$. 
Since $m+T$ is not an integer, $m+T$ is not zero, and by Assumption~\ref{nontorsionAs}, $1-q^{m+T}$ is nonzero, and so is the leading coefficient $d$ of $\gamma'$. Then $\dcmH_m(W)\neq 0$. In the right-hand side of the equation $\dcmH_m(U)=\dcmH_m(V)+\dcmH_m(W)$, none of the terms in $\dcmH_m(V)$ can cancel the nonzero $\dcmH_m(W)$ because, by~\eqref{vanishH} and~\eqref{raiseBl}, $\dcmH_m(V)\in\bigoplus_{i=a+1}^{b}\gradsub_i$, while $\dcmH_m(W)\in\gradsub_{b+1}$. Therefore, $\dcmH_m(U)\neq 0$.
The case $\psi=\dcmTH_m$ is proven similarly. This completes the proof for $t=1$. 

Suppose the statement holds for any composition of $n<t$ elements of $\nonZdcm$. We write $\psi=\zeta\circ\phi$ for some $\zeta\in\nonZdcm$ and some composition $\phi$ of $t-1$ elements from $\nonZdcm$. By the inductive hypothesis, for any nonzero $U\in\Heisen$, the element $V:=\phi(U)$ is nonzero. Then by the inductive hypothesis, $0\neq \zeta(V) = \psi(U)$. Therefore, the statement holds for any value of $t$.
\end{proof}

\def\nonzeroId{\mathcal{T}}
\begin{theorem}\label{IdThm} The algebra $\Protoq$ faithfully represents $\Proto$.
\end{theorem}
\begin{proof}
Showing that $\Proto$ is simple will suffice. To do this, we prove that, given a~(two-sided) ideal $\nonzeroId$ of $\Proto$, the following are equivalent.
\begin{enumerate}
    \item\label{IdT} $\nonzeroId\neq 0$.
    \item\label{IdCk} $C^M \in\nonzeroId$ for some positive integer $M$.
    \item\label{IdK} $K\in\nonzeroId$.
\item\label{IdInvert} $\nonzeroId=\Proto$.
\end{enumerate}

The equivalence between~\ref{IdK} and~\ref{IdInvert} follows from the invertibility of $K$. To see why the implication~\ref{IdCk}~$\Rightarrow$~\ref{IdK} is true, simply consider the fact that $K=K^{1+2h}C^h$ since $K^2 C=1$ and $K$ commutes with $C$. Since $K$ is a~basis element of $\Proto$, it is nonzero, and so we immediately have~\ref{IdK}~$\Rightarrow$~\ref{IdT}. We now show~\ref{IdT}~$\Rightarrow$~\ref{IdCk}. Suppose $\nonzeroId$ is a~nonzero ideal of $\Proto$, and choose a~nonzero $x\in\nonzeroId$. If we write $x$ as a~linear combination of the basis elements~\eqref{protoBasis} of $\Proto$, then each basis element is of the form $K^{2h+i}U$ for some $h\in\N$, some $i\in\{0,1\}$ and some basis element $U$ of $\Heisen$ from~\eqref{ABCbasis}. Let $N$ be the maximum possible $h$ among the basis elements $K^{2h+i}U$ that appear in the linear combination. Since $K^2 C=1$ and $CK=KC$, by multiplying $C^N$ to the left of $x$, every term on the right-hand side becomes $C^N K^{2h+i}U=C^k K^i U=K^i C^k U$, for some $k\in\N$. Observe that $C^k U$ is also another basis element of $\Heisen$ from~\eqref{ABCbasis}. Since $i$ is either $0$ or $1$, then
\begin{eqnarray}
C^N x=f+Kg,\label{CNx}
\end{eqnarray}
for some $f,g\in\Heisen$. If $f=0$, then we multiply both sides of~\eqref{CNx} by $CK$, and we obtain $KC^{N+1}x=g$. Here, $g$ cannot be zero for this will contradict $x\neq 0$. Thus, we use Corollary~\ref{magicCor} on $g$, and we are done. Suppose $f\neq 0$. If $g=0$, then we can again make use of Corollary~\ref{magicCor}, and we are done for this case. Thus, we now assume that both $f$ and $g$ are nonzero. Using Corollary~\ref{magicCor} on $f$, let $\psi$ be the composition of elements of $\Zdcm$ such that, for some nonzero scalar $c_1$ and positive integer $H$, we have $\psi(f)=c_1C^H$. Thus, $\psi(C^N x)=c_1C^H +\psi(Kg)$, and by~\eqref{dcmzero},
\begin{eqnarray}
(\dcmTH_H\circ\psi)(C^N x)=(\dcmTH_H\circ\psi)(Kg).\label{CNx2}
\end{eqnarray}
Note here that $H\in\Z$ and so we can write $(\dcmTH_H\circ\psi)$ as
\begin{eqnarray}
(\dcmTH_H\circ\psi) = \zeta_1\circ\zeta_2\circ\cdots\circ\zeta_r,\nonumber
\end{eqnarray}
for some $r\in\Zplus$ and some $\zeta_1,\zeta_2,\ldots,\zeta_r\in\Zdcm$. Consequently,
\begin{eqnarray}
(\dcmTH_H\circ\psi) (Kg)= \left(\zeta_1\circ\zeta_2\circ\cdots\circ\zeta_r\right)(Kg).\label{toQ0}
\end{eqnarray}
Using the identities in Proposition~\ref{KdcmProp} repeatedly on the right-hand side of~\eqref{toQ0}, there exists $N\in\Q$ and some $\xi_1,\xi_2,\ldots,\xi_r\in\nonZdcm$ such that
\begin{eqnarray}
(\dcmTH_H\circ\psi) (Kg)= q^N K\left(\xi_1\circ\xi_2\circ\cdots\circ\xi_r\right)(g).\label{toQ}
\end{eqnarray}
Since $g\neq 0$, by Proposition~\ref{nonzeroQdcmProp}, $g':=\left(\xi_1\circ\xi_2\circ\cdots\circ\xi_r\right)(g)$ is a~nonzero element of $\Heisen$. Multiplying both sides of~\eqref{toQ} by $q^{-N}KC$, we have $q^{-N}KC(\dcmTH_H\circ\psi) (Kg)=g'$, and so~\eqref{CNx2} becomes
\begin{eqnarray}
q^{-N}KC(\dcmTH_H\circ\psi)(C^N x)=g'.\label{CNx3}
\end{eqnarray}
We apply Corollary~\ref{magicCor} to the nonzero element $g'$ of $\Heisen$, and so there exists a~composition $\varphi$ of elements of $\Zdcm$, some nonzero scalar $c_2$, and some positive integer $M$ such that~\eqref{CNx3} becomes
\begin{eqnarray}
c_2^{-1}\varphi(q^{-N}KC(\dcmTH_H\circ\psi)(C^N x)) = C^M.\label{CNx4}
\end{eqnarray}
Since $x$ is an element of the ideal $\nonzeroId$, which is invariant under deformed commutator mappings, the left-hand side of~\eqref{CNx4} is an element of $\nonzeroId$. Therefore, $C^M \in\nonzeroId$.
\end{proof}

By Proposition~\ref{embeddingProp}, we immediately have:
\begin{corollary}\label{embedHCor} The subalgebra of $\Proto$ generated by $A,B$ is isomorphic to $\Heisen$.
\end{corollary}
Hence, we identify $\anni$ as $A$ and $\crea$ as $B$. All the properties of $\Heisen$ are then satisfied in the subalgebra of $\Proto$ generated by $A,B$.

\begin{remark}\label{longRem} We rewrite in here the relations~\eqref{combase1} to~\eqref{combase4} in terms of $A$, $B$, and $C=\LieAB$:

\begin{eqnarray}
\LieABiso^{k+2} & = & \frac{-q^k (1-q)}{1-q^{k+1}}\sum_{i=0}^k \frac{\left(\left(\ad \creaIso\right)\circ\left(-\ad \LieABiso\right)^{k}\circ(\ad A)\right)\left(C\right)}{(q-1)^{1+i}},\label{combaseIso1}\\
\LieABiso^{k+1}\anniIso^l & = & -\frac{\left(\left(-\ad \LieABiso\right)^k \circ\left(-\ad \anniIso\right)^{l+1}\right)\left(\creaIso\right)}{(1-q)^l (q^l -1)^k}, \label{combaseIso2}\\
\LieABiso^{k+1}\pcreaIso^l & = & q^{l(k+1)}\frac{\left(\left(\ad \creaIso\right)^{l-1}\circ\left(\ad \LieABiso\right)^{k+1}\right)\left(B\right)}{(q-1)^{k+1}(1-q^{k+1})^{l-1}},\label{combaseIso4}
\end{eqnarray}
where $k\in\N$ and $l\in\Zplus$. The relation~\eqref{combaseIso4} is~\eqref{combase4} with the adjustment from $B^l C^k$ to $C^k B^l$ as discussed in Section~\ref{prelSec}. All such relations~\eqref{combaseIso1} to~\eqref{combaseIso4} hold in $\Proto$. Additionally, because of the embedding of $\Heisen$ in $\Proto$ as given in Corollary~\ref{embedHCor}, by Remark~\ref{LieBasisRem}, the elements
\begin{eqnarray}
C^k, & & \quad\quad\quad\quad (k\in\Zplus),\label{ctypeLieH}\\
A,\quad C^k A^l, & & \quad\quad\quad\quad (k,l\in\Zplus),\label{ctypeA0H}\\
B,\quad C^k B^{l}, & & \quad\quad\quad\quad (k,l\in\Zplus),\label{ctypeB0H}
\end{eqnarray}
form a~basis for the Lie subalgebra of $\Proto$ generated by $A,B$ (or by $A,B,C$), while any (finite) linear combination or
\begin{eqnarray}
1,\quad A^l,\quad B^l,\quad\quad\quad (l\in\N\backslash\{0,1\}),\label{nonLiebasisH}
\end{eqnarray}
is not a~Lie polynomial in in $A,B$ (nor in $A,B,C$).

\end{remark}

\section{Commutator algebra for the generators \texorpdfstring{$K,A,B,C$}{K,A,B,C} of \texorpdfstring{$\Proto$}{P}}\label{commSec}

In this section we find a description for the commutator algebra or Lie subalgebra of $\Proto$ generated by $K,A,B,C$. 
In the next two lemmas, we exhibit some important properties of the basis elements~\eqref{protoBasis}.

\begin{lemma}\label{LieBasisLem} Every basis element in~\eqref{protoBasis}, except the multiplicative identity $\algI$, is a~Lie polynomial in $K,A,B,C$.
\end{lemma}
\begin{proof}
For convenience in later computations, we use the notation
\begin{eqnarray}
p:=q^{\frac{1}{2}}.
\end{eqnarray}
Thus, the defining relations of $\Proto$ become
\begin{eqnarray}
AB & = & \frac{1-p^2 C}{1-p^2},\label{prel1}\\
BA & = & \frac{1-C}{1-p^2},\label{prel2}\\
AC & = & p^2 CA,\label{prel3}\\
BC & = & p^{-2}CB,\label{prel4}\\
AK & = & p^{-1}KA,\label{prel5}\\
BK & = & pKB,\label{prel6}\\
CK & = & KC,\label{prel7}\\
K^2 C & = & 1.\label{prel8}
\end{eqnarray}
By some routine computations and induction, the relations~\eqref{prel3} to~\eqref{prel6} can be generalized into:
\begin{eqnarray}
A^l K^h & = & p^{-hl}K^h A^l,\label{reorderAK}\\
B^l K^h & = & p^{hl} K^h B^l,\label{reorderBK}\\
A^l C^k & = & p^{2kl} C^k A^l,\label{reorderAC}\\
B^l C^k & = & p^{-2kl} C^k B^l,\label{reorderBC}
\end{eqnarray}
which hold for any $h,k,l\in\N$. Our strategy is to segregate the basis elements $U\neq\algI$ among~\eqref{protoBasis} into the following:
\begin{eqnarray}
C^k A^l,\quad C^k B^l,\quad C^k, & \quad\quad (k,l\in\Zplus),\label{LieH}\\
KC^k A^l,\quad KC^k B^l, & \quad\quad (k,l\in\Zplus)\label{LieHreorder},\\
K^h A^l,\quad K^h B^l,\quad K^l,\quad KC^l, & \quad\quad\quad\quad (h\in\N,\ l\in\Zplus).\label{compuBasis}
\end{eqnarray}
By~\eqref{combase1} to~\eqref{combase4}, the basis elements~\eqref{LieH} are all elements of $\LieProto$. Using the reordering formulas~\eqref{reorderAK} and~\eqref{reorderBK}, and the commutativity of $K$ with $C$ in some routine computations, we have the relations
\begin{eqnarray}
KC^k A^l = \frac{p^l}{1-p^l}\lbrak C^k A^l,K\rbrak,\label{basistype2A}\\
KC^k B^l = \frac{1}{1-p^l}\lbrak K,C^k B^l \rbrak,\label{basistype2B}
\end{eqnarray}
which hold for any $k,l\in\Zplus$. But since $C^k A^l,\ K,\ C^k B^l \in\LieProto$ for any $k,l$, the relations~\eqref{basistype2A} and~\eqref{basistype2B} imply that the basis elements~\eqref{LieHreorder} are all in $\LieProto$. We now sketch the computations necessary to prove that basis elements~\eqref{compuBasis} are in $\LieProto$. First, we consider those of the form $K^h A^l,K^h B^l$ for the case $h\in\Zplus$. Using induction and the reordering formulas~\eqref{reorderAK} and~\eqref{reorderBK}, the relations
\begin{eqnarray}
K^h A^l & = & \frac{p^{hl}}{(1-p)^l (1-p^l)^{h-1}}\left(\left(-\ad K\right)^{h-1}\circ\left(\ad A\right)^l \right)(K),\quad (h,l\in\Zplus),\label{hlimited1}\\
K^h B^l & = & \frac{1}{(1-p)^l (1-p^l)^{h-1}}\left(\left(\ad K\right)^{h-1}\circ\left(-\ad A\right)^l \right)(K),\quad (h,l\in\Zplus),\label{hlimited2}
\end{eqnarray}
hold in $\Proto$. In view of Remark~\ref{adRem}, the relations~\eqref{hlimited1},~\eqref{hlimited2} imply that $K^h A^l,K^h B^l \in\LieProto$ whenever $h,l\in\Zplus$.

We are missing the case $h=0$ in the relations~\eqref{hlimited1} and~\eqref{hlimited2}, and so we turn our attention to some other types of Lie polynomials in $K,A,B,C$ on which we can use the reordering formulas~\eqref{reorderAK} to~\eqref{reorderBC} and also the defining relations of $\Proto$ in order to show that $A^l,B^l \in\LieProto$. Using the reordering formulas~\eqref{reorderAK} and~\eqref{reorderBK}, we have
\begin{eqnarray}
p^{l+1}\lbrak KB,KA^{l+1}\rbrak = p^{l+2}K^2 (BA)A^l -K^2 A^l (AB).\label{technique1}
\end{eqnarray}
Using~\eqref{prel1},~\eqref{prel2}, we replace $BA$ and $AB$ in the right-hand side of~\eqref{technique1}. This gives us
\begin{eqnarray}
p^{l+1}(1-p^2)\lbrak KB,KA^{l+1}\rbrak = K^2 (p^{l+2}(1-C)A^l -A^l (1-p^2 C)).\label{technique2}
\end{eqnarray}
By further manipulations that also make use of the relation~\eqref{prel8}, we turn~\eqref{technique2} into
\begin{eqnarray}
A^l = \frac{1-p^2}{p(1-p^l)}\lbrak KA^{l+1},KB\rbrak-\frac{1-p^{l+2}}{p^{l+2}(1-p^l)}K^2 A^l,\quad\quad\quad(l\in\Zplus).\label{LieAl}
\end{eqnarray}
We have established previously that $KA^{l+1},\ KB,\ K^2 A^l \in\LieProto$, and so~\eqref{LieAl} asserts that $A^l \in\LieProto$. By a~computational pattern similar to that done in~\eqref{technique1} to~\eqref{LieAl}, the relations
\begin{eqnarray}
B^l & = & \frac{1-p^2}{1-p^l}p^{l-1}\lbrak KA,KB^{l+1}\rbrak-\frac{1-p^{l+2}}{1-p^l}p^{l-2}K^2 B^l,\quad\quad\quad(l\in\Zplus),\label{LieBl}\\
K^{h+2} & = & \frac{1-p^2}{1-p^{h+2}}\lbrak K^{h+2}A,B\rbrak+\frac{1-p^{h}}{1-p^{h+2}}p^{2}K^h,\quad\quad\quad\quad\quad\quad(h\in\N),\label{LieKh}\\
KC^l & = & \frac{1-p^2}{1-p^{2l}}p^{2l-2}\lbrak KC^{l-1}A,B\rbrak+\frac{1-p^{2l-2}}{1-p^{2l}}KC^{l-1},\quad\quad\quad\ \ (l\in\Zplus),\label{LieKCl}
\end{eqnarray}
can be shown to hold in $\Proto$ by routine calculations. As established earlier, $KA$, $KB^{l+1}$, $K^2 B^l \in\LieProto$. Then by~\eqref{LieBl}, $B^l \in\LieProto$.

So far, the relations~\eqref{hlimited1},\eqref{hlimited2},\eqref{LieAl},\eqref{LieBl} assert that all basis elements in~\eqref{compuBasis} of the form $K^h A^l,\ K^h B^l$ (with $h\in\N$, $l\in\Zplus$) are in $\LieProto$. For those of the form $K^l$, we use~\eqref{LieKh}. If $h=0$ in~\eqref{LieKh}, then we find a~relation asserting that $K^2$ is a~scalar multiple of $\lbrak K^2 A,B\rbrak$, where $K^2 A\in\LieProto$ as previously established. Thus, $K^2 \in\LieProto$. We can use the condition $K,K^2 \in\LieProto$ and~\eqref{LieKh} in an inductive argument which proves that $K^l \in\LieProto$ for all $l\in\Zplus$. By a~similar argument,~\eqref{LieKCl} can be used to show that $KC^l \in\LieProto$ for any $l\in\Zplus$. This completes all the cases needed in the proof.
\end{proof}

\begin{lemma}\label{closureLem} If $U,V$ are any two basis elements from~\eqref{protoBasis} neither of which is the multiplicative identity $\algI$, and if $\lbrak U,V\rbrak=\sum_{i=1}^t c_tW_t$ for some nonzero scalars $c_t$ and some basis elements $W_t$ from~\eqref{protoBasis}, then $W_t\neq 1$ for any $t$.
\end{lemma}
\begin{proof}
Since $\lbrack U,V\rbrack =0$ if one of $U$, $V$ is $1$, we further assume that $U$ and $V$ are of the form $K^h C^k B^l$ or $K^h C^k A^l$ for some $h,k\in\N$ and some $l\in\Zplus$. If $U,V\in\Heisen$ (i.e., if $h=0$), then by Remark~\ref{longRem}, we find that it is not possible for $W_t$ to be $1$ for any $t$. Thus, we further assume the exponent of $K$ in at least one of $U$, $V$ is positive. 

It is routine to show, using~\eqref{reorderAK},~\eqref{reorderBK}, that under all such remaining cases, either $\lbrack U,V\rbrack =K^N \lbrack U',V'\rbrack$ or $\lbrack U,V\rbrack =K^N W'$ for some positive integer $N$, where $U'$, $V'$, $W'$ are among~\eqref{ctypeLieH} to~\eqref{ctypeB0H}. If $\lbrack U,V\rbrack =K^N \lbrack U',V'\rbrack$, then by Remark~\ref{longRem}, $\lbrack U',V'\rbrack$ is also one of~\eqref{ctypeLieH} to~\eqref{ctypeB0H}, but $K^N$ multiplied by any one of~\eqref{ctypeLieH} to~\eqref{ctypeB0H} is not $1$. Neither is $K^N W'$ for the last remaining case.
\end{proof}

\begin{theorem}\label{LieThm} The Lie subalgebra $\LieProto$ of $\Proto$ generated by $K,A,B,C$ is precisely the span of all the basis elements~\eqref{protoBasis} except the multiplicative identity $\algI$. i.e., The algebra $\Proto$ can be decomposed as the direct sum
\begin{eqnarray}
\Proto = \C\algI\oplus\LieProto.
\end{eqnarray}
\end{theorem}
\begin{proof}
Let $\mathcal{V}$ be the span of all basis elements~\eqref{protoBasis} except $\algI$. Thus, $\mathcal{V}$ contains the generators $K,A,B,C$ of $\LieProto$. Lemma~\ref{closureLem} asserts that $\mathcal{V}$ is a~Lie subalgebra of $\Proto$, while Lemma~\ref{LieBasisLem} implies $\mathcal{V}\subseteq\LieProto$. By the minimality of $\LieProto$ among all Lie subalgebras of $\Proto$ that contain $K,A,B,C$, we have $\mathcal{V}=\LieProto$.
\end{proof}

\section{On a~\texorpdfstring{$q$}{q}-oscillator representation of the Askey-Wilson algebra}

We now discuss how our algebraic results are related to a~Lie algebra structure in some representation of an Askey-Wilson algebra. In particular, we shall consider the following.
\begin{definition}
The Askey-Wilson algebra, denoted by $\AW$, is the unital associative algebra over $\C$ with generators $\genzero$, $\genone$, $\gentwo$ and relations
\begin{eqnarray}
e^\freq\genzero\genone-e^{-\freq}\genone\genzero & = & \gentwo,\label{AW1}\\
e^\freq\gentwo\genzero-e^{-\freq}\genzero\gentwo & = & b\genzero+c_1\genone+d_1,\label{AW2}\\
e^\freq\genone\gentwo-e^{-\freq}\gentwo\genone & = & b\genone+c_0\genzero+d_0,\label{AW3}
\end{eqnarray}
where $\freq\in\R\backslash\{0\}$. The scalars $b,c_0,c_1,d_0,d_1\in\C$ are called the \emph{structure constants} of $\AW$.
\end{definition}
According to~\cite[pp.~1155--1156]{Zhe91}, if the conditions
\begin{eqnarray}
b=c_1=d_0=d_1=0,\qquad c_0=e^{2\freq}-1,\label{AWstruct}
\end{eqnarray}
are imposed on the structure constants of $\AW$, then the algebra $\AW$ has a~$q$-oscillator representation given by
\begin{eqnarray}
\genzero & = & e^{\freq\Num},\label{K0def}\\
\genone & = & \anni+\crea.\label{K1def}
\end{eqnarray}
In this representation, $\genone$ is said to be the $\freq$-analogue of the position operator.
\begin{definition}
Define $\AWq$ as the subalgebra of $\Protoq$ generated by the representations of $\genzero$ and $\genone$ given in~\eqref{K0def},\eqref{K1def}, with $q=e^{-2\freq}$.
\end{definition}
Thus, the choice here of $q$ is nonzero and not a~root of unity, and the results of the previous sections apply. Recall our notation $p=q^{\frac{1}{2}}$. By $q=e^{-2\freq}$ and~\eqref{AWstruct}, the relations~\eqref{AW1}-\eqref{AW3} can be rewritten as
\begin{eqnarray}
p^{-1}\genzero\genone-p\genone\genzero & = & \gentwo,\label{K2def}\\
p^{-1}\gentwo\genzero-p\genzero\gentwo & = & 0,\nonumber\\
p^{-1}\genone\gentwo-p\gentwo\genone & = & \frac{1-p^2}{p^2}\genzero,\nonumber
\end{eqnarray}
By~\eqref{K2def}, the algebra $\AWq$ contains the representation of $K_2\in\AW$ as an operator on the $q$-oscillator. By Theorem~\ref{IdThm}, we can represent elements of $\AWq$ in terms of the generators $K,A,B,C$ of $\Proto$. Thus,~\eqref{K0def},~\eqref{K1def} become
\begin{eqnarray}
\genzero & = & K,\nonumber\\
\genone & = & A+B.\nonumber
\end{eqnarray}
As for the generator $\gentwo$, we use~\eqref{K2def} and the reordering formulas~\eqref{proto5},~\eqref{proto6}. By some routine calculations, we have
\begin{eqnarray}
\gentwo = p^{-1}(1-p)KA + p^{-1}(1-p^3)KB.\label{gentwoREL}
\end{eqnarray}
Our next goal is to determine the relationship between the algebras $\Proto$ and $\AWq$, and as shall be seen in the computations and proofs that we will present, the Lie polynomials in the generators of these algebras play an important role.
\begin{lemma}\label{AWtoPLem} The Lie subalgebra $\LieAWq$ of $\AWq$ generated by $\genzero,\genone,\gentwo$ is a~Lie subalgebra of $\LieProto$.
\end{lemma}
\begin{proof}
Using~\eqref{prel5},~\eqref{prel6}, and by some routine calculations, we obtain
\begin{eqnarray}
KA & = & \frac{p\lbrak A,K\rbrak}{1-p},\\
KB & = & \frac{\lbrak K,B\rbrak}{1-p}. \label{BKLie}
\end{eqnarray}
From~\eqref{gentwoREL}-\eqref{BKLie}, we have enough information to express $\gentwo$ as a~Lie polynomial in $K,A,B$ by routine computations. To summarize all relevant equations, we have
\begin{eqnarray}
\genzero & = & K,\label{KLie0}\\
\genone & = & A+B,\label{KLie1}\\
\gentwo & = & \lbrak A-p^{-1}(1+p+p^2)B,K\rbrak.\label{KLie2}
\end{eqnarray}
An immediate consequence of~\eqref{KLie0}-\eqref{KLie2} is that every Lie polynomial in $\genzero,\genone,\gentwo$ is also a~Lie polynomial in $K,A,B$, which are enough to generate $\LieProto$ since $C=\lbrak A,B\rbrak$.
\end{proof}
The most natural continuation is to consider the converse: is it possible to write any Lie polynomial in the generators of $\Proto$ as a~Lie polynomial in the generators of $\AWq$? We now show that we have an answer in the affirmative.
\begin{lemma}\label{PtoAWLem} The generators $K,A,B$ of the Lie algebra $\LieProto$ are Lie polynomials in $\genzero,\genone,\gentwo$.
\end{lemma}
\begin{proof}
Since $\genzero=K$, we only need to show $A,B$ are Lie polynomials in $\genzero,\genone,\gentwo$. We start with evaluating $\lbrak\genzero,\genone\rbrak$ using~\eqref{KLie0},\eqref{KLie1} and the reordering formulas~\eqref{proto5},\eqref{proto6}. The result is a~linear combination of $KA$ and $KB$, and so we can use~\eqref{gentwoREL} to solve for $KA$ and $KB$. From these routine computations, we have
\begin{eqnarray}
KA & = & \frac{p^2 \gentwo+p(1+p+p^2)\lbrak\genone,\genzero\rbrak}{(1-p)(1+p)^2},\label{LieKA}\\
KB & = & \frac{p\gentwo+p\lbrak\genzero,\genone\rbrak}{(1-p)(1+p)^2}.\label{LieKB}
\end{eqnarray}
The right-hand sides of~\eqref{LieKA},\eqref{LieKB} are Lie polynomials in $\genzero,\genone,\gentwo$, and so we have $KA,KB\in\LieAWq$. Our next step is to compute the commutator of these new Lie polynomials $KA$, $KB$ with $\genone=A+B$. Using routine computations that make use of the reordering formulas~\eqref{prel1},\eqref{prel2},\eqref{prel5},\eqref{prel6}, we obtain
\begin{eqnarray}
KA^2 -\frac{p^2}{1-p^2}KC & = & \frac{p}{1-p}\lbrak \genone,KA\rbrak +\frac{p}{1-p^2}\genzero,\label{elimA}\\
KB^2 -\frac{1}{1-p^2}KC & = &\frac{\lbrak KB,\genone\rbrak}{1-p}+\frac{\genzero}{p(1-p^2)}.\label{elimB}
\end{eqnarray}
We also compute the commutator of $\genone=A+B$ with $\gentwo$ as given in~\eqref{gentwoREL}. Moreover, using~\eqref{prel1},\eqref{prel2},\eqref{prel5},\eqref{prel6}, we obtain, by routine calculations, the relation
\begin{eqnarray}
KA^2 -p(1+p+p^2)KB^2 +p\frac{1+p^2}{1-p^2}KC=\frac{p^2}{(1-p)^2}\lbrak\genone,\gentwo\rbrak-\frac{1+p^2}{1-p^2}\genzero.\label{elimAB}
\end{eqnarray}
Because $KA$, $KB$ are Lie polynomials in $\genzero,\genone,\gentwo$, then so are the right-hand sides of~\eqref{elimA}-\eqref{elimAB}. Denote these elements of $\LieAWq$ by $X$, $Y$, $Z$, respectively. We can eliminate $KA^2$ and $KB^2$ from~\eqref{elimAB} using~\eqref{elimA},\eqref{elimB}. By routine computations, we have the relation
\begin{eqnarray}
KC=\frac{1+p}{p}X-(1+p)(1+p+p^2)Y-\frac{1+p}{p}Z\quad\in\ \ \LieAWq.\nonumber
\end{eqnarray}
We have thus shown that $KA$, $KB$, $KC$ are Lie polynomials in $\genzero,\genone,\gentwo$. The significance of this is that by some routine computations that involve the reordering formulas~\eqref{prel1}-\eqref{prel8}, we have
\begin{eqnarray}
A & = & \frac{\lbrak KC,KA\rbrak}{1-p}\\
B & = & \frac{p\lbrak KB,KC\rbrak}{1-p}.
\end{eqnarray}
Therefore, $A,B\in\AWq$ as desired.
\end{proof}

\begin{theorem}
The algebras $\AWq$ and $\Proto$ are equal, and so are the corresponding Lie algebras $\LieAWq$ and $\LieProto$.
\end{theorem}
\begin{proof}
By Lemma~\ref{PtoAWLem}, the generators $K,A,B$ and even $C=\lbrak A,B\rbrak$ of $\LieProto$ are in $\LieAWq$, and so we have the inclusion $\LieProto\subseteq\LieAWq$. Using Lemma~\ref{AWtoPLem}, we further have $\LieProto=\LieAWq$. By Lemma~\ref{LieBasisLem} each basis element $U\neq 1$ of $\Proto$ from~\eqref{protoBasis} is a~Lie polynomial in $K,A,B,C$, and so according to Lemma~\ref{PtoAWLem}, $U\in\LieAWq\subseteq\AWq$. The remaining basis element $1=K^2 C$ is also in $\AWq$ since $K,C\in\LieAWq\subseteq\AWq$. Therefore, $\Proto$ is contained in its subalgebra $\AWq$, and is hence equal to $\AWq$.
\end{proof}

\section*{Acknowledgement} This research was supported by De La Salle University, Manila, through the Research and Grants Management Office, with grant number 15FU1TAY20-1TAY21.



{\small\bibliography{commat}}

\EditInfo{January 17, 2023}{June 27, 2023}{Adam Chapman and Mohamed Elhamdadi}

\end{document}